\documentclass[12pt,reqno, oneside]{amsart}

\usepackage[utf8]{inputenc}
\usepackage[normalem]{ulem}
\usepackage{textcomp}
\usepackage[english]{babel}
\usepackage{amsmath, amsfonts, amssymb,amsthm}
\usepackage{bbold}
\usepackage[dvipsnames]{xcolor}
\usepackage{upgreek}
\usepackage{float}
\usepackage{changepage}
\usepackage{array}
\usepackage{enumerate}
\usepackage{youngtab}
\usepackage{hyperref}
\usepackage{enumitem}

\usepackage{verbatim}
\usepackage{tikz-cd}
\usepackage{tikz}
\usepackage{tikz-cd}
\usetikzlibrary{decorations.shapes}
\usepackage{fancyhdr}
\usepackage{caption}

\usepackage{mathtools}
\mathtoolsset{showonlyrefs}

\usepackage{tikz}
\usetikzlibrary{ decorations.markings}
\usetikzlibrary{calc, matrix, arrows,decorations.pathmorphing, positioning}
\numberwithin{equation}{section}

\newtheorem{theorem}{Theorem}[section]
\newtheorem{lemma}[theorem]{Lemma}
\newtheorem{prop}[theorem]{Proposition}

\newtheorem{cor}[theorem]{Corollary}
\newtheorem{defi}[theorem]{Definition}
\newtheorem{remark}[theorem]{Remark}
\newtheorem{example}[theorem]{Example}

\newcommand{\C}{\mathbb{C}}
\newcommand{\Z}{\mathbb{Z}}
\newcommand{\Zp}{\mathbb{Z}_{\geq 0}}

\newcommand{\td}{\tilde}
\newcommand{\ci}{\circ}

\newcommand{\sL}{\mathfrak{sl}_2}
\newcommand{\sLh}{\widehat{\mathfrak{sl}}_2}

\newcommand{\Ker}{\mathrm{Ker}}
\newcommand{\h}{\mathfrak{h}}

\newcommand{\bs}{\boldsymbol}

\newcommand{\Wr}{\mathop{\mathrm{Wr}}}

\newcommand{\pd}{\partial}

\newcommand{\al}{\alpha}

\newcommand{\la}{\lambda}

\newcommand{\be}{\begin{equation}}
\newcommand{\ee}{\end{equation}}
\newcommand{\benonum}{\begin{equation*}}
\newcommand{\eenonum}{\end{equation*}}
\newcommand{\bsme}{\begin{subequations}}
\newcommand{\ese}{\end{subequations}}
\newcommand{\ba}{\begin{aligned}}
\newcommand{\ea}{\end{aligned}}
\newcommand{\beanonum}{\begin{align*}}
\newcommand{\eeanonum}{\end{align*}}

\DeclareMathSymbol{\shortminus}{\mathbin}{AMSa}{"39}

\newlength\mylen
\makeatother
\title{Monodromy free Schr\"odinger operators and  
affine $\sL$ master functions.}

\begin{document}
\author{Andrei Grigorev and Evgeny Mukhin} 
\address{EM: Department of Mathematical Sciences,
Indiana University Indianapolis,
402 N. Blackford St., LD 270, 
Indianapolis, IN 46202, USA;}
\email{emukhin@iu.edu} 

\address{AG: Department of Mathematical Sciences,
Indiana University Indianapolis,
402 N. Blackford St., LD 270, 
Indianapolis, IN 46202, USA
;}
\email{aagrigor@iu.edu, andrei.al.grigorev@gmail.com}
\begin{abstract}
Given a non-zero polynomial $P(x)$, we study Fuchsian differential operators of the form $L=\partial_x^2-u(x)$ such that for all $\la\in\C$ the operator $L+\la P(x)$ is monodromy free. We prove that all such operators are obtained from populations of critical points of $\sLh$ master functions.
Moreover, we show that the reproduction procedure  of critical points corresponds to a Darboux transformation of operator $P^{-1}(x)L$. As a result, we obtain a classification of all operators $L$ with such properties in the case of $P(x)=x^k$.

\end{abstract}
\maketitle
\section{Introduction}
In the classical works, \cite{AM78}, \cite{AMM77}, the rational tau functions of the KdV equations were described by the Fuchsian operators of the form $L=\partial^2-u(x)$, where $u(x)$ is a rational function, with the property that $L+\la$ is monodromy free for all $\la\in\C$. Moreover, it was shown in \cite{DG86} that all such operators are obtained from $\partial^2$ by a finite sequence of Darboux transformations.

The operators of the form $L=\partial^2-x^2-u(x)$, where $u(x)$ is a rational function decreasing at $\infty$, with the property that $L+\la$  is monodromy free for all $\la\in\C$ were classified in \cite{O99}. All such operators are obtained from $\partial^2-x^2$ by a finite sequence of Darboux transformations. These operators were used to study Calogero systems with quadratically grown  potentials, see \cite{FV94}. There is no known classification of similar operators with $x^2$ replaced by $x^k$, $k\neq 0,2$, see \cite{GV09}.

The interest to operators with similar properties spiked dramatically after the  the seminal work \cite{BLZ96}.  In the study of quantum KdV flows, \cite{BLZ96} introduced the operators of the form $\partial^2_x-u(x)+\la x^{\bar k}$, where $u(x)$ is a rational function, $\bar k\in\C$, with the property that for all $\la\in\C$ the monodromy around poles of $u(x)$ is trivial. Such operators became known as Schr\"odinger operators with monster potential. The correspondence between generating function of eigenvalues of quantum KdV Hamiltonians and connection functions (in variable $\lambda$) of the operators with monster potential is named by the ODE/IM correspondence.

In \cite{MMR25} it was argued that the connection functions of interest in the operators with monster potential coincide with the connection functions of another class of operators of the form $\partial^2-u(x)+\lambda x^k(x-1)$ which have no monodromy for all $\lambda$, and which are Fuchsian for $\lambda=0$. 

In this paper, we fix a non-zero polynomial $P(x)$ and consider Fuchsian operators of the form $L=\partial^2-u(x)$, such that $L+\lambda P(x)$ are monodromy free for all $\lambda$. (More generally, we actually allow monodromy to equal $\pm \mathbb{1}$.) We call such operators $\lambda$-monodromy free.

\medskip

Our main tool is the populations of critical points of master functions. From that point of view, the case of \cite{AM78} corresponds to the trivial representation of $\sLh$, the case of \cite{O99} to the Fock module of $\widehat{\mathfrak{gl}}_1$, the case of \cite{MMR25} to a tensor product of $\sLh$ vacuum level one module at $x=0$ with a Verma module of level $k\in \C$ at $x=1$, while the case of the present paper corresponds to an arbitrary tensor product of $\sLh$ integrable representations located at zeroes of $P(x)$.

We sketch the relation between $\la$-monodromy free operators and populations of critical points more precisely. Let $T_0(x), T_1(x)$ be non-zero polynomials, and let $P=T_0T_1$. Then  zeroes of a pair of polynomials $(y_0(x), y_1(x))$ is an $\sLh$ critical point (a.k.a. a solution of the Gaudin Bethe ansatz equations, \eqref{eq:Bethe_ansatz})  if there exist polynomials $\tilde y_0,$ $\tilde y_1$ such that  
\begin{equation} \label{eq:introWr}
\Wr(y_0,\tilde y_0)=T_0 y_1^2, \qquad \Wr(y_1,\tilde y_1)=T_0 y_1^2,
\end{equation}
and if pairs $(y_0,y_1)$ , $(y_0,T_0)$, $(y_1,T_1)$ have no common roots. The populations are connected components of the closure of the variety of critical points and consist of pairs $(y_0,y_1)$ satisfying \eqref{eq:introWr} but some roots could be common. The changes $(y_0(x),y_1(x))\to (\tilde y_0(x),y_1(x)) $ and $(y_0(x),y_1(x))\to (y_0(x),\tilde y_1(x))$ preserve the populations and are called reproductions, see \cite{MV04} and Section \ref{sec:populations}.
 
Given $(y_0,y_1)$ in a population, the operators
$$L_0=P^{-1}\Big(\partial +\ln'\Big(\frac{y_0^2}{T_0^{1/2}y_1^2}\Big)\Big)\Big(\partial -\ln'\Big(\frac{y_0^2}{T_0^{1/2}y_1^2}\Big)\Big),\qquad 
L_1=P^{-1}\Big(\partial +\ln'\Big(\frac{y_1^2}{T_1^{1/2}y_0^2}\Big)\Big)\Big(\partial -\ln'(\frac{y_1^2}{T_1^{1/2}y_0^2})\Big)
$$
are $\lambda$-monodromy free, see Proposition \ref{prop:la_mf_fromcp}. Conversely, we prove that all $\lambda$-monodromy free operators are obtained this way, and this relation is almost a bijection, see Theorem \ref{thm:DO_to_populations}.

Such a correspondence allows us to transfer the knowledge about populations of critical points to the $\la$-monodromy free operators.  

\medskip 

Factorize a $\la$-monodromy free operator $L$ in the linear factors, $L=D_1D_2$, where $D_i=P^{-1/2}\partial+g_i(x)$. Then we call the transformation $L=D_1D_2\to \tilde L=D_2D_1$ the Darboux transformation. We show that the Darboux transformation of a $\la$-monodromy free operator is $\la$-monodromy free and that it corresponds to the reproduction, namely $L_0(y_0,y_1)\to L_1(\tilde y_0, y_1)$ and    $L_1(y_0,y_1)\to L_0(y_0, \tilde y_1)$ are Darboux transformations.

In the case of $T_0=x^{2l}, T_1=x^{2m}$, $2l,2m\in\Z_{\geq 0}$, the population of critical points is unique and it is generated by the trivial critical point $(y_0(x),y_1(x))=(1,1)$, see \cite{MV14}. 
As the result we obtain that $\la$-monodromy free operators with $P=x^k$, $k\in\Z_{>0}$ are exactly the operators obtained by a finite sequence of Darboux transformations from the operators $$x^{-k}\Big(\partial^2-\frac{m(m+1)}{x^2}\Big),\qquad m=0,\frac12,1,\frac 32,\dots, \frac12{\left\lfloor{\frac{k}{2}}\right\rfloor},$$ see Theorem \ref{thm:op_clas}. 

The case $k=0$ corresponds to the Adler--Moser case \cite{AM78} which is related to the KdV hierarchy. We wonder if other populations of critical points also correspond to integrable hierarchies. We give an affirmative answer in the case of $T_0=T_1$ and the population generated by the trivial critical point. This case corresponds to a simple change of variables in the KdV hierarchy, see Section \ref{sec:KdV}.

\medskip

The Adler--Moser case has been generalized to the case of differential operators of order $n$, see\cite{VW14}. That corresponds to the trivial $\widehat{\mathfrak{sl}}_n$-module.
We plan to address the case of tensor products of integrable $\widehat{\mathfrak{sl}}_n$-modules in a future publication. 

\medskip

The paper is constructed as follows. We discuss critical points, populations and the reproduction procedure in Section \ref{sec:populations}. We pay special attention to the population  originated at the trivial critical point. Section \ref{sec:DO} is dedicated to the study of $\la$-monodromy free operators and their Darboux transformations. We connect the populations and $\la$-monodromy operators in Section \ref{sec:Pop_and_DO}. Finally, in Section \ref{sec:KdV} we give a change of variables which relates a population generated by the trivial point in the case of $T_0=T_1$ to the Adler--Moser polynomials and describe the corresponding change of the celebrated KdV flows.

\section{Populations}\label{sec:populations}
In this section we recall the notions of $\sLh$ master function, the corresponding Bethe ansatz equations (BAE), reproductions and critical points, following \cite{MV04}. We study the populations which contain the trivial critical point in more detail.

\subsection{Master functions}

A master function is a multi-variable function which  depends on a generalized Cartan matrix $A$, a collection of dominant integral weights for the corresponding Lie algebra $\mathfrak{g}(A)$, and a positive point in the root lattice of the corresponding root system $R(A)$.

In this work we study $\sLh$  master functions. Set $A = \begin{pmatrix}
    2&-2\\ -2 &2
\end{pmatrix}$. Then $\mathfrak{g}(A)=\sLh$.  Fix the standard basis $\{\al_0^{\vee},\al^{\vee}_1,d\}$ in the Cartan subalgebra $\h \subset \sLh$ and the dual basis $\{\omega_0, \omega_1, \delta\}\subset \h^*$. Let $\rho = \omega_0 + \omega_1 + \delta$,  $\al_1=2\omega_1-2\omega_0$, $\al_0=\delta-\al_1$.

Let $\mathcal P=\Z \omega_0 \oplus \Z\omega_1\oplus \Z\delta$  and
$\mathcal{P}_{+} = \Z_{\geq0} \omega_0 \oplus \Z_{\geq0}\omega_1\oplus \Z_{\geq0}\delta$ be the lattice of integral weights and the cone of dominant integral weights, respectively. 

The Weyl group $W(A) \cong \Z \rtimes \Z_2$ is generated by reflections $\mathrm{s}_0,\mathrm{s}_1$. The Weyl group acts on $\mathcal{P}$ by  $\mathrm{s}_i(\mu) = \mu - \langle \mu, \al_i^{\vee}\rangle \al_i$.  The shifted action of the Weyl group is given by $w\cdot \mu = w(\mu + \rho) -\rho$.

Let $\bs z=(z_1,\dots,z_n) \in \C^n$ be a sequence of distinct complex numbers.
Let $\bs \mu=(\mu_1,\dots\,\mu_n)$, where $\mu_r=2l_r\omega_0+2m_r\omega_1  \in \mathcal{P}_+\backslash\{0\}$, $r=1,\dots,n$, be a set of non-zero dominant integral weights.
 
\begin{defi}
    The $\sLh$ master function corresponding to $\bs z$ and $\bs \mu$ is a rational function of variables  $\bs{s}=\{s_1,\dots, s_{d_0}\},\;\bs{t}= \{t_1,\dots, t_{d_1}\}$ given by
$$
\Phi(\bs{s},\bs{t}) = \prod_{1\leq i<i^{\prime}\leq d_0}(s_i-s_{i^{\prime}})^2\prod_{1\leq j<j^{\prime}\leq d_1}(t_j-t_{j^{\prime}})^2 \prod_{\substack{1\leq i \leq d_0\\ 1\leq j\leq d_1}}(s_i-t_j)^{-2}\prod_{\substack{1\leq i \leq d_0\\ 1\leq r \leq n}}(s_i-z_r)^{-2l_r}\prod_{\substack{1\leq j \leq d_1\\ 1\leq r \leq n}} (t_j-z_r)^{-2m_r}.
$$
\end{defi}

We call $(\bs{s},\bs{t})$ a critical point of $\Phi(\bs{s},\bs{t})$ if
$
\partial_{s_i}\ln \Phi = 0, \partial_{t_j}\ln\Phi = 0,
$ 
for all $1\leq i \leq d_0,\; 1\leq j\leq d_1$. Then $(\bs{s},\bs{t})$ is a critical point if and only if

\begin{equation} \label{eq:Bethe_ansatz}
\begin{aligned}
\sum_{i^{\prime}:i^{\prime}\neq i}\frac{1}{s_i-s_{i^{\prime}}} - \sum_{j=1}^{d_1}\frac{1}{s_i-t_j} - \sum_{r=1}^{n}\frac{l_r}{s_i-z_r} &= 0, \qquad 1\leq i \leq d_0,\\
    \sum_{j^{\prime}:j^{\prime}\neq j}\frac{1}{t_j-t_{j^{\prime}}} - \sum_{i=1}^{d_0}\frac{1}{t_j-s_i} - \sum_{r=1}^n\frac{m_r}{t_j-z_r} &= 0, \qquad 1\leq j \leq d_1.\\
\end{aligned}
\end{equation}

This system is called the system of Bethe ansatz equations.

\subsection{Populations}\label{subsec:populations}
We encode the data $\bs z$ and $\bs \mu$ into a pair of polynomials $\bs T(x)=(T_0(x),T_1(x))$ and the pair $(\bs s,\bs t)$ into a pair of polynomials $\bs y = (y_0,y_1)$ of a formal variable $x$:

$$T_0(x) = \prod_{r=1}^n(x-z_r)^{2l_r},\; T_1(x) = \prod_{r=1}^n(x-z_r)^{2m_r},\;\; y_0(x) = \prod_{1\leq i \leq d_0}(x-s_i),\; y_1(x) = \prod_{1\leq j \leq d_1}(x-t_j).$$ 

We extend $T_i$ and $y_i$ for all $i\in \Z$ periodically, by setting $T_i=T_j,\; y_i=y_j$ whenever $i-j$ is even.

A pair $\bs{y}$ is called generic with respect to $\bs{T}$ if $y_{0}$ and $y_{1}$ are multiplicity free, coprime, and if $y_i$ are coprime with $T_i$ for $i = 0,1$.

We are interested only in the zeroes of polynomials $y_0,y_1$ and therefore we consider them up to multiplication by a non-zero complex constant.

\begin{prop}\cite{MV04}\label{prop:cp_wr}
  The pair $\bs y$ corresponds to a critical point of master function if and only if $\bs y$ is generic and there exist polynomials $\tilde{y}_0, \tilde{y}_1$ such that  
  
  \begin{equation}\label{eq:cp_wr}
            \Wr(y_0,\tilde{y}_0) = y_1^2T_0,\;\Wr(y_1,\tilde{y}_1) = y_0^2T_1.
        \end{equation}
Moreover, if either pair $(\tilde  y_0, y_1)$ or $(y_0, \tilde y_1)$ is generic then  it also corresponds to a critical point.
\end{prop}\qed

Note that if $\bs y$ corresponds to a critical point of a master function, then  $(y_0, \td{y}_1 + c_1y_1)$ and $(\td{y}_0 + c_0 y_0 , y_1)$ are generic for all but finitely many $c_0, c_1 \in \C$.
By Proposition \ref{prop:cp_wr},  such pairs correspond to critical points of a master function associated to the same $\bs{T}$ but a different number of variables $d_0,d_1$.

We call a pair $\bs y= (y_0,y_1)$ fertile if there exist polynomials $\tilde{y}_0, \tilde{y}_1$ such that Equation \eqref{eq:cp_wr} is satisfied.

\begin{defi}\label{defi:reproduction}
    For a fertile pair $(y_0, y_1)$  we say that the pairs $\bs{y}^{[0]} = (y_0^{[0]},y_1^{[0]}):=(\td{y}_0,y_1),\;\bs{y}^{[1]} = (y_0^{[1]}, y_1^{[1]}):=(y_0,\td{y}_1)$ are obtained from the pair $(y_0,y_1)$ by reproduction in the directions $0$ and $1$, respectively.
\end{defi}

\begin{defi} Let $\bs y$ correspond to a critical point. The set of all pairs of polynomials obtained from $\bs y$ by iterated reproductions is called the population of $\bs y$.
\end{defi}
Fertile pairs of polynomials $\bs y$ such that any pair of polynomials obtained from $\bs y$ by a sequence of reproductions is again fertile are called super-fertile.
Note that if $\bs y$ corresponds to a critical point, then it is super-fertile.  In any population, the pairs $\bs y$ corresponding to critical points are dense, see \cite{MV04}. In particular any point in a population is super-fertile.

Note that if $\bs{y}$ is a fertile pair, then the set of pairs of polynomials obtained from $\bs{y}$ by two reproductions in the same direction is the same as the set of pairs of polynomials obtained from $\bs{y}$ by applying only one reproduction in this direction. 

The degrees of polynomials in pairs belonging to a population of $\bs y$ are described by the shifted action of the Weyl group.

\begin{prop}\cite{MV04}
    Let $\bs{y} = (y_0, y_1)$ be fertile and let $i\in\{0,1\}$. Assume that $\deg(y_i^{[i]}) \neq \deg(y_i)$. Then 
    $$
    \sum_{r=1}^n\mu_r - \sum_{j=0}^1\deg(y^{[i]}_j)\al_j=\mathrm{s}_i\cdot \Big(\sum_{r=1}^n\mu_r - \sum_{j=0}^1\deg(y_j)\al_j\Big).
    $$
    \qed
\end{prop}

\subsection{Fertile and super-fertile pairs of polynomials}
In this section, we study fertile and super-fertile pairs of polynomials.

\begin{lemma}\label{lemma:sf_from_population}
    Let $\bs{y}= (y_0,y_1)$ and $\bs{T} =(T_0,T_1)$ be pairs of polynomials. Let $f_0,f_1$, be non-zero monic polynomials such that $\bs{T}^\circ =  \left(\frac{T_0f_0^2}{f_1^2},\frac{T_1f_1^2}{f_0^2}\right)$ are polynomials. Let $\bs{y}^\ci = (f_0y_0,f_1y_1)$. Then 
    \begin{itemize}
        \item If $\bs{y}$ is fertile with respect to $\bs{T}$, then $\bs{y}^\ci$ is fertile with respect to $\bs{T}^{\ci}$.
        \item If $\bs{y}$ is super-fertile with respect to $\bs{T}$, then $\bs{y}^\ci$ is super-fertile with respect to $\bs{T}^\ci$.
    \end{itemize}
    Moreover, a pair $(\bs{y}^{\circ})^{[i]}$ is obtained by a reproduction  from $\bs{y}^{\circ}$ if and only if  $(\bs{y}^{\circ})^{[i]}=(f_0y_0^{[i]}, f_1y_1^{[i]})$, where $\bs{y}^{[i]}=(y_0^{[i]}, y_1^{[i]})$ is obtained by a reproduction from $\bs{y}$.
\end{lemma}
\begin{proof}
    The lemma follows from the identities
    $$
    \Wr(f_0y_0, f_0\tilde{y}_0) = (f_1y_1)^2T_0^{\ci},\qquad \Wr(f_1y_1, f_1\tilde{y}_1) = (f_0y_0)^2T_1^{\ci}.
    $$
\end{proof}

We have a converse of Lemma \ref{lemma:sf_from_population}.

\begin{prop}\label{prop:population_from_sf}
    Let $\bs{y}^{\circ}=(y_0^{\ci},y_1^{\ci})$ be a super-fertile pair of polynomials corresponding to $\bs{T}^\ci = (T_0^{\ci},T_1^{\ci})$. Then there exist unique monic polynomials $f_0,f_1$ such that 
    \begin{itemize}
        \item $\bs{T} = \left(\frac{T_0^{\ci}f_1^2}{f_0^2},\frac{T_1^{\ci}f_0^2}{f_1^2}\right)$ is a pair of polynomials.
        \item There exists a population of critical points corresponding to $\bs{T}$  and a unique point $\bs{y} = (y_0,y_1)$ in this population satisfying $\bs{y}^{\ci}=(f_0 y_0,f_1y_1)$.
    \end{itemize}
\end{prop}
\begin{proof}
We first prove the existence of $f_0,f_1,\bs{y}$.

It is sufficient to prove the statement for any pair $\hat{\bs{y}}$ obtained from $\bs{y}^{\circ}$ by a finite sequence of reproductions, see Lemma \ref{lemma:sf_from_population}.

First, applying a finite sequence of reproductions to $\bs{y}^{\ci}$ we obtain a pair $\bs{\hat{y}}$ such that for $i=0,1$ and any choice of 
$\hat{y}_i^{[i]}$ we have
\begin{equation}\label{eq:sfp_pf1}
    \mathrm{ord}_{z}(\hat{y}_i^{[i]}) \geq \mathrm{ord}_{z}(\hat{y}_i),\textit{ for all } z\in \C \textit{ such that } \mathrm{ord}_{z}(\hat{y}_i)>1.
\end{equation}
Indeed, consider the vector space generated by polynomials $y_i^{\ci},(y_i^{\ci})^{[i]}$. Due to \eqref{eq:cp_wr} any non-zero linear combination of $y_i^{\ci}$ and $(y_i^{\ci})^{[i]}$ may vanish up to order greater than $1$ only at zeroes of $y_{i+1}^{\circ}T_{i}^{\circ}$. For any zero $z$ of $y_{i+1}^{\circ}T_{i}^{\circ}$, if $\mathrm{ord}_z((y_i^{\ci})^{[i]})>\mathrm{ord}_z(y_i^{\ci})$, then for almost any $c\in\C$ we have $\mathrm{ord}_z((y_i^{\ci})^{[i]}+cy_i^{\ci}) = \mathrm{ord}_z(y_i^{\ci})$. Therefore, for an appropriate choice of $(y_i^{\ci})^{[i]}$  we have $\mathrm{ord}_z((y_i^{\ci})^{[i]})\leq \mathrm{ord}_z(y_i^{\ci})$ for all $z$ such that $\mathrm{ord}_z((y_i^{\circ})^{[i]})>1$. This shows that for any $z$ such that $\mathrm{ord}_z(y^{\ci}_i)>1$ a generic reproduction in the $i$-th direction does not increase $\mathrm{ord}_z(y_i^{\ci})$ and decreases it if  $\mathrm{ord}_{z}((y^{\ci}_i)^{[i]}) < \mathrm{ord}_{z}(y^{\ci}_i)$.

 From \eqref{eq:sfp_pf1} and \eqref{eq:cp_wr}, for any zero $z$ of $\hat{y}_{i+1}T_i^{\ci}$ we have
 \begin{equation}\label{eq:ineq_sf}
\mathrm{ord}_{z}(T_i^{\ci}) + 2(\mathrm{ord}_z(\hat{y}_{i+1})-\mathrm{ord}_z(\hat{y}_{i}))\geq 0.
\end{equation}

Denote $$\displaystyle f_i = \prod_{\substack{ z\in\C:\ \hat{y}_{i+1}(z)T_i^{\circ}(z)=0}}(x-z)^{\mathrm{ord}_z(\hat{y}_i)}.$$ By construction $y_i= \frac{\hat{y}_i}{f_i}$, $\tilde y_i=\frac{\hat{y}_i^{[i]}}{f_i}$  are polynomials and inequalities \eqref{eq:ineq_sf} imply that $\bs{T} = (T_0,T_1) =\left(\frac{T_0^{\ci}f_1^2}{f_0^2},\frac{T_1^{\ci}f_0^2}{f_1^2}\right)$ is a pair of polynomials. In addition, \eqref{eq:cp_wr} is satisfied. Therefore the pair $\bs{y}=(y_0, y_1)$ is fertile. This pair is generic since by construction $y_i$ and $T_{i}y_{i+1}$ have no common zeroes. This implies that the pair $\bs{y}$ represents a critical point corresponding to $\bs{T}$.

It remains to show the uniqueness part. The polynomials $f_i$ can be recovered from $\bs{y}^{\circ}$ as the greatest common divisors of all $i$-th components of pairs obtained from $\bs{y}^{\ci}$ by reproductions. Therefore $\bs{y}, \bs{T}$ are unique.
\end{proof}

\begin{cor}\label{cor:}
    Let pairs of monic polynomials $(y_0,y_1)$ and $(\bar{y}_0,\bar{y}_1)$ be in (possibly different) populations of critical points corresponding to the same $\bs{T} = (T_0,T_1)$. If $\frac{y_0}{y_1} = \frac{\bar y_0}{\bar y_1}$, then $(y_0,y_1) = (\bar{y}_0, \bar{y}_1)$.
\end{cor}
\begin{proof}
    There exist coprime polynomials $f,g$ such that $f y_0 = g\bar{y}_0$. This implies that $f y_1 = g\bar{y}_1$. Denote $\bs{y}^{\circ} = (fy_0,fy_1) = (g\bar{y}_0,g\bar{y}_1)$. Then $\bs{y}^{\circ}$ is a super-fertile pair corresponding to $\bs{T}^{\circ} =\bs{T}$ and both $(y_0,y_1), (\bar{y}_0,\bar{y}_1)$ satisfy the conditions of Proposition \ref{prop:population_from_sf}. Therefore $(y_0,y_1)= (\bar{y}_0,\bar{y}_1)$.
\end{proof}

Proposition \ref{prop:population_from_sf} says that all super-fertile pairs of polynomials $\bs y^\ci$ are obtained via multiplication of pairs $\bs y$ from a population of critical points. Division of  $\bs y$ (whenever possible) produces fertile (but not super-fertile) pairs.  

Let $\bs y=(y_0,y_1)$ be in a population of critical points corresponding to $\bs T=(T_0,T_1)$. Let $f_0, f_1$ be rational functions such that 
$\bs y^\ci=(f_0y_0,f_1y_1)$, $\tilde{\bs y}^\ci=(f_0y_0^{[0]},f_1y_1^{[1]})$, 
$T^\ci=(\frac{T_0f_0^2}{f_1^2}, \frac{T_1f_1^2}{f_0^2)})$ 
are all pairs of polynomials.

\begin{lemma} \label{divide fertile lem}
The pair of polynomials $\bs y^\ci$ is fertile with respect to $\bs T^\ci$.

If $\bs y^\ci$ is super-fertile with respect to $\bs T^\ci$, then $f_0$ and $f_1$ are polynomials.
\end{lemma}
\begin{proof}
    The first statement follows as in Lemma \ref{lemma:sf_from_population}. The second statement follows from the fact that the greatest common divisor of, say, first coordinates of elements of a population is one and therefore cannot be divided by any polynomial.
\end{proof}
We do not know if all fertile pairs of polynomials are obtained as in Lemma \ref{divide fertile lem}.

\subsection{Population of \texorpdfstring{$\bs y=(1,1)$}{}}\label{subsec:triv_pop}
The populations of trivial critical points play an important role in this paper, see Sections \ref{subsec:from_lmf_to_pop} and \ref{sec:KdV}.  In this section we write coordinates of a point of the population of the trivial critical point as a Wronski determinant of polynomials satisfying a simple recursion, see \eqref{eq:theta_Wr} below.

Fix a pair of monic polynomials $\bs T = (T_0, T_1)$. Set $P(x) = T_0(x) T_1(x)$.

For $i=0,1$, define differential operators
$$L_i =P^{-1} \Big(\pd - \frac{1}{2}\ln'(T_i)\Big)\Big(\pd + \frac{1}{2}\ln'(T_i)\Big).$$
Here and everywhere $\pd=\pd_x$ is the derivative with respect to $x$. Also here and below we denote $\ln'(f) = \frac{\pd f}{f}$.

For $i = 0, 1$, we define a sequence of functions $\{\psi_{n}^{(i)}\}_{n\in\Z_{>0}}$ by 
\begin{equation}\label{eq:psi_rec}
\psi^{(i)}_1 = \frac{1}{\sqrt{T_i}},\qquad 
{L}_i\psi^{(i)}_{n+1} = \psi_{n}^{(i)}, \ n > 0.
\end{equation}

The functions $\psi_n^{(i)}$ depend on $2n-2$ integration constants. Note that $ L_i\psi_1^{(i)} = 0$.

Set 
$$\phi^{(i)}_{n} = \sqrt{T_i}\psi_n^{(i)}.$$
We have $\phi_1^{(i)}=1$.
\begin{prop}\label{prop:phi_rec}
    For $n\in\Z_{>0}$ and $i=0,1$, functions $\phi^{(i)}_{n}$ satisfy the relation
    $$
    \phi^{(i)}_{n+1} = \int T_i\int T_{i+1} \phi_{n}^{(i)}dx.
    $$
    In particular, $\phi_n^{(i)}$ are polynomials.
\end{prop}
\begin{proof}

    Define polynomials $\tilde{\phi}_n^{(i)}$ by $\tilde\phi_1^{(i)} = 1$ and $\tilde{\phi}_{n+1}^{(i)} = \int T_i\int T_{i+1} \td{\phi}_{n}^{(i)}dx$ for $n>0$ . We show that functions $\tilde{\psi}_n^{(i)} = \frac{\tilde{\phi}_{n+1}^{(i)}}{\sqrt{T_i}}$ satisfy \eqref{eq:psi_rec} . Indeed, 
    \begin{equation}
        L_i\Big(\frac{\tilde{\phi}_{n+1}^{(i)}}{\sqrt{T_i}}\Big) = \frac{1}{P\sqrt{T_i}}\Big(\pd - \ln'(T_i)\Big)\pd \tilde\phi_{n+1}^{(i)} = \frac{1}{P\sqrt{T_i}}\Big(\pd - \ln'(T_i)\Big)T_i\int T_{i+1} \tilde\phi_{n}^{(i)}= \frac{\sqrt{T_i}T_{i+1}}{P}\tilde\phi_{n}^{(i)} = \frac{\tilde\phi_n^{(i)}}{\sqrt{T_i}}.
    \end{equation}
\end{proof}

Let $$a_{2r}= r^2, \qquad a_{2r+1}= r(r+1).$$
Note that for all $n$
\begin{equation}\label{imya?} 
{2a_{n+1}-2a_n=\begin{cases} n,  & n=2r,  \\n+1,   & n=2r+1, \end{cases} } \qquad   a_{n+1}-a_{n-1} = n.
\end{equation}

For $n\in\Z_{> 0}$ and $i = 0,1$,  define 
\begin{equation}\label{eq:theta_Wr}
    \theta^{(i)}_n := \frac{\Wr(\phi_1^{(i)},\dots, \phi_n^{(i)})}{T_i^{a_{n}}T_{i-1}^{a_{n-1}}}.
\end{equation}
We also set $\theta_0^{(i)}=1$.

Note that even though the sequence $(\phi_1^{(i)},\dots, \phi_n^{(i)})$ depends on $2n-2$ parameters, the parametric family $\theta^{(i)}_n$ depends only on $n-1$ of them. Indeed,  changing the second integration constant in $\phi^{(i)}_{j}$ results in the change  $\phi^{(i)}_{s} \mapsto \phi^{(i)}_s + c\phi_{s+1-j}^{(i)}$, $s=j,j+1,\dots,n$ while $\phi^{(i)}_{s}$, $s=1,\dots,j-1$, remain the same. The Wronskian \eqref{eq:theta_Wr} clearly does not change.

\begin{prop} \label{prop:AM_gen_recc}
 Functions $\theta_{n}^{(i)}$ are polynomials which satisfy 
\begin{equation}\label{eq:AM_gen_recc}
\Wr(\theta^{(i)}_{n-1}, \theta^{(i)}_{n+1}) = T_{n+i}\cdot(\theta^{(i)}_{n})^2.
\end{equation}
The set of all pairs $(\theta^{(i)}_{2r+1+i},\theta^{(i)}_{2r+i})$ and $(\theta^{(i)}_{2r+1-i},\theta^{(i)}_{2r+2-i})$ for all $r \in \Z_{\geq 0}, i = 0,1$
 coincides with the population of the critical point $(1,1)$ corresponding to $\bs T$.
\end{prop}
\begin{proof}
    The proof is based on the identities of Wronskians and it  is similar to the proof of Lemma $2.2$ in \cite{dV17}.
    
    We repeatedly use the following elementary properties of the Wronskians which hold for any polynomials $g_1,\dots, g_n$ and $g$,
$$
\Wr(1,g_1,\dots, g_n)=\Wr(g_1',\dots,g_n'), \qquad \Wr(g g_1,g g_2,\dots, g g_n)=g^{n}\Wr(g_1,\dots,g_n).
$$
     In addition, for any polynomials $f_1,\dots, f_{n+1}, \chi$, we have  the identity,
    \begin{equation}\label{eq:Wr_th_pf_1}
        \Wr(\Wr(f_{1},\dots, f_{n}, \chi), \Wr(f_{1},\dots, f_{n}, f_{n+1})) = \Wr(f_{1},\dots, f_{n})\Wr(f_{1},\dots, f_{n},\chi,f_{n+1}).
    \end{equation}

    Set $f_j = \phi^{(i)}_j$ for all $j = 1,\dots, n+1$ and $\chi =\int T_i=\theta_2^{(i)}$. Recall that $\phi_1^{(i)} = 1$. 
    Then, using Proposition \ref{prop:phi_rec}, for $n\in\Z_{>0}$, we obtain
    \begin{multline}\label{eq:Wr_th_pf_2}
        \Wr(\phi^{(i)}_{1},\dots, \phi^{(i)}_{n}, \chi) = \Wr((\phi^{(i)}_2)',\dots, (\phi^{(i)}_n)', T_i) = \Wr\Big(T_{i}\int T_{i+1}\phi^{(i)}_1,\dots, T_i\int T_{i+1}\phi^{(i)}_{n-1}, T_i\Big) =\\= T_i^n\Wr\Big(\int T_{i+1}\phi^{(i)}_1,\dots, \int T_{i+1}\phi^{(i)}_{n-1},1\Big) = (-1)^{n-1}T_i^{n+a_{n{-1}}}T_{i+1}^{n-1+a_{n{-2}}}\theta^{(i)}_{n-1}.
    \end{multline}
    Substituting \eqref{eq:Wr_th_pf_2} into \eqref{eq:Wr_th_pf_1}, we obtain
    $$
    \Wr((-1)^{n-1}T_i^{n+a_{n-1}}T_{i+1}^{n-1+a_{n-2}}\theta^{(i)}_{n-1}, T_i^{a_{n+1}}T_{i+1}^{a_{n}}\theta^{(i)}_{n+1}) = (T_i^{a_{n}}T_{i+1}^{a_{n-1}}\theta^{(i)}_n)((-1)^{n+1}T_i^{n+1+a_{{n}}}T_{i+1}^{n+a_{n{-1}}}\theta^{(i)}_{n}).
    $$
    {Using \eqref{imya?}, we see that the powers of $T_i$ and $T_{i+1}$ in the two arguments of the Wronskian are the same, and we get}
$$\Wr(\theta^{(i)}_{n-1},\theta^{(i)}_{n+1}) = T_i^{{2a_n-2a_{n+1}} +n+1} T_{i+1}^{{2a_{n-1}-2a_{n}}+n}(\theta^{(i)}_n)^2{=T_{n+i}\ (\theta^{(i)}_n)^2.}$$

Therefore the pairs $(\theta^{(i)}_{2r+1-i},\theta^{(i)}_{2r+2-i})$ and $(\theta^{(i)}_{2r+1+i},\theta^{(i)}_{2r+i})$ are in the population for any choice of integration constants in the recursion \eqref{eq:psi_rec}.

It remains to show that we obtain all points of the population. Fix some $i = 0, 1$. We proceed by induction on $n$. 
Note that  $\Ker(L_{i})$ is spanned by $\frac{1}{\sqrt{T_i}}, \frac{1}{\sqrt{T_i}}\chi$. This implies that given $\phi^{(i)}_s$, $s=1,\dots,n-1$, the function $\phi^{(i)}_n$ is defined up to a change $\phi^{(i)}_n \to \phi^{(i)}_n + c_1 + c_2 \chi$. The Wronskian does not depend on $c_1$ since $\phi^{(i)}_1=1$.
Let $\phi^{(i)}_n(c) =  \phi^{(i)}_n + (-1)^nc \,\chi$. Using \eqref{eq:Wr_th_pf_2}, we obtain
$$
\frac{\Wr(\phi_1^{(i)},\dots,\phi_{n}^{(i)}(c))}{T_i^{a_n}T_{i+1}^{a_{n+1}}}= \theta_n^{(i)} + (-1)^nc\frac{(-1)^nT_{i}^{n-1+a_{n-2}}T_{i+1}^{n-2+a_{n-3}}\theta^{(i)}_{n-2}}{T_i^{a_n}T_{i+1}^{a_{n-1}}} = \theta_{n}^{(i)} + c\, \theta_{n-2}^{(i)}.
$$
By induction hypothesis, the pairs $(\theta_{n-2}^{(i)} ,\theta_{n-1}^{(i)})$ for $n+i$ odd and $(\theta_{n-1}^{(i)} ,\theta_{n-2}^{(i)})$ for $n+i$ even, cover all the pairs in the population obtained via the corresponding sequence of reproductions. Addition of one reproduction corresponds to adding the pairs $(\theta_{n}^{(i)}(c),\theta_{n-1}^{(i)})$ for $n+i$ odd and $(\theta_{n-1}^{(i)} ,\theta_{n}^{(i)}(c))$ for $n+i$ even, where $\theta_n^{(i)}(c) = \theta_{n}^{(i)} + c \theta_{n-2}^{(i)}$ and $\theta_{n-2}^{(i)}$ runs through all possible choices in the part of the population obtained on the previous step.
\end{proof}
The relation of the sequences $\{\theta_n^{(i)}\}_{n\in\Zp}$ to the population of $(1,1)$ can be illustrated by the following picture. In this picture we denote by $\mathrm{s}_i$ the reproduction procedure in direction $i$. The top two pairs correspond to $n=0$ and the critical point $(1,1)$. The two pairs below correspond to $n=1$, the next two to $n=2$, etc.

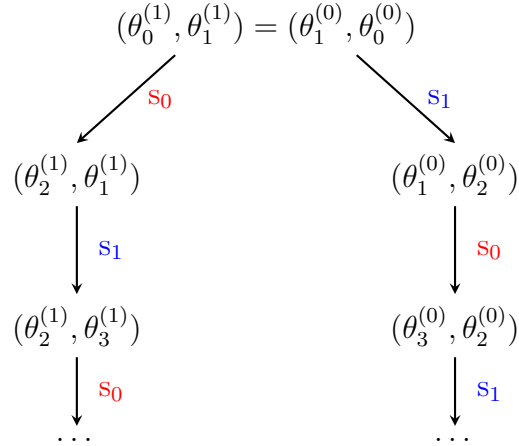
\begin{center}
\begin{tikzpicture}[>=stealth, thick]
    \node (root) at (0, 0) {$ (\theta_0^{(1)}, \theta_1^{(1)}) = (\theta_1^{(0)}, \theta_0^{(0)})$};

    \node (L1) at (-2.5, -2) {$(\theta_2^{(1)}, \theta_1^{(1)})$};
    \node (L2) at (-2.5, -4) {$(\theta_2^{(1)}, \theta_3^{(1)})$};
    \node (L3) at (-2.5, -5.5) {$\dots$};

    \node (R1) at (2.5, -2) {$(\theta_1^{(0)}, \theta_2^{(0)})$};
    \node (R2) at (2.5, -4) {$(\theta_3^{(0)}, \theta_2^{(0)})$};
    \node (R3) at (2.5, -5.5) {$\dots$};

    \draw[->] ([xshift=-1.2cm]root.south) -- (L1.north) node[midway, right=4pt, text=red] {$\mathrm{s}_0$};
    \draw[->] ([xshift=1.2cm]root.south) -- (R1.north) node[midway, right=4pt, text=blue] {$\mathrm{s}_1$};

    \draw[->] (L1.south) -- (L2.north) node[midway, right=4pt, text=blue] {$\mathrm{s}_1$};
    \draw[->] (L2.south) -- (L3.north) node[midway, right=4pt, text=red] {$\mathrm{s}_0$};

    \draw[->] (R1.south) -- (R2.north) node[midway, right=4pt, text=red] {$\mathrm{s}_0$};
    \draw[->] (R2.south) -- (R3.north) node[midway, right=4pt, text=blue] {$\mathrm{s}_1$};
\end{tikzpicture}
\end{center}

\section{Differential operators} \label{sec:DO}
In this section we define and study (non-monic) $\la$-monodromy free (Fuchsian, rational) Schr\"odinger operators. Then we recall the notion of Darboux transformations for Schr\"odinger operators and introduce a generalized version for the case of non-monic Schr\"odinger operators. We  relate the Darboux transformations to 
the reproduction procedure.

\subsection{$\la$-monodromy free Schr\"odinger operators}
Let $P\in \C[x]\backslash\{0\},\; U\in \C(x)$, let $L = \frac{1}{P}(\pd^2 - U)$ be a non-monic Schr\"odinger operator. We call $U$ the potential of $L$.

Operator $L$ is called Fuchsian if the potential $U$ has poles of order at most $2$ and $x U(x)\to 0$ as $x\to \infty$.

Following \cite{DG86, O99, GV09}, we study operators $L$ such that all eigenfunctions of $L$ have no monodromy. We increase this family of operators by allowing monodromy $-\mathbb{1}$. 

\begin{defi}
  We call the operator $L=\frac1P(\partial^2-U)$  a $\la$-monodromy free operator if for any $\la\in\C$ and any $s\in\C$, 
the monodromy of operator $\partial^2-U-\la P$ around $x=s$ 
is $\pm \mathbb{1}$.   
\end{defi}

The property of $L$ being $\la$-monodromy free is equivalent to an algebraic system of equations on the coefficients of the expansion of $U$ at poles. Let $s$ be a pole of $U$. Let the expansion of $U-\la P$ at $s$ be $\displaystyle U-\la P=\sum_{j=-2}^{\infty}a_j(s,\la)(x-s)^j$. Note that $a_j(s,\la)$ are linear in $\la$ and 
\begin{equation}\label{eq:ajs}
    a_{{-}2}(s,\la) {=} a_{{-}2}(s),\; a_{-1}(s,\la) {=} a_{{-}1}(s),\; a_0(s,\la) {=} a_0(s,0) {+} \la P(s),\; a_{1}(s,\la) {=} a_1(s,0) {+} \la P'(s).
\end{equation}
The necessary and sufficient conditions for the potential $U$ so that $L$ is $\la$-monodromy free are given by the following standard proposition, see for example \cite{I56}[Ch. 16]. 
\begin{prop}\label{prop:mf_det}
The operator $L=\frac1P(\partial^2-U)$ is $\la$-monodromy free if and only if for any pole $s$ of the potential $U$,  we have
\begin{enumerate}
    \item $a_{-2}(s) = m_{s}(m_{s}+1)$, where
    $m_s\in \frac{1}{2}\Z_{>0}$. 
    \item The polynomial $\Delta(s,\la)$ in $\la$ given by

\begin{equation}\label{eq:no_monodromy_det}
\Delta(s,\la)=\det
\begin{pmatrix}
a_{-1}(s) & 1\cdot 2m_{s} & 0 & 0 & \cdots & 0 \\
a_0(s,\la) & a_{-1}(s) & 2\cdot(2m_{s}-1) & 0 & \cdots & 0 \\
a_1(s,\la) & a_{0}(s,\la) & a_{-1}(s) & 3\cdot(2m_{s}-2) & \cdots & 0 \\
 \\
\vdots & \ddots & \ddots & \ddots & \ddots & 0 \\
\vdots & \ddots & \ddots & a_0(s,\la) & a_{-1}(s) & 2m_{s}\cdot1 \\
a_{2m_{s}-1}(s,\la) & \cdots & & a_{1}(s,\la) & a_{0}(s,\la) & a_{-1}(s)
\end{pmatrix}
\end{equation}
vanishes for all $\lambda$, $\Delta(s,\la)= 0$.
\end{enumerate}
\end{prop}
\begin{proof}

Using Frobenius method one looks for solutions in the form of a power series, 
$$(x-s)^{\al}\Big(1+\sum_{j=1}^{\infty}b_j(x-s)^{j}\Big),$$
where $\al \in \{-m_s,m_s+1\}$ are the two solutions of the characteristic equation. A solution with $\al=m_s+1$ always exists. The existence of a solution with $\al=-m_s$ is equivalent to the vanishing of $\Delta(s,\la)$.
\end{proof}
Proposition \ref{prop:mf_det} defines $m_s \in \frac{1}{2}\Z_{\geq0}$ for any pole $s$ of $U$. We set $m_s =0$ for all $s\in \C$, where $U$ is regular. Then $-m_s$, $m_s+1$ are exponents of $L$ at $x=s$. We call $m_s$ the modified exponent of $L$ at $x=s$.

\subsection{Residues of potentials of $\la$-monodromy free operators}
We show that monodromy of a $\la$-monodromy free operator is trivial (cannot be $-\mathbb{1}$) at poles of the potential $U$, which are not zeroes of $P$ and compute the residues of $U$ at these poles.

\begin{prop}\label{prop:mf_residue}
    Let $s\in\C$ and $P(s)\neq 0$.  If for all $\la\in\C$, $\Delta(s,\la)=0$, then $m_s\in \Z_{\geq 0}$ and
    \begin{equation}\label{eq:residue_formula}
        a_{-1}(s) = \frac{m_s(m_s+1)P'(s)}{2P(s)}.
    \end{equation}
\end{prop}
\begin{proof}
    We claim that the polynomial $\Delta(s,\la)$ has the form

    \begin{align}
        \Delta(s,\la) &= \frac{(-1)^{3m_s + \frac{1}{2}}((2m_s)!)^2P(s)^{ m_s + \frac{1}{2}}}{2^{m_s -\frac{1}{2}}( m_s-\frac{1}{2})! (2m_s-1)!!}\la^{ m_s + \frac{1}{2}} + \dots \  &m_s&\in \Z+\frac12,\;\;\;\;\;\label{eq:det_asympt_odd} \\
        \Delta(s,\la) &=  \frac{(-1)^{m_s}P(s)^{m_s}((2m_s)!)^2}{((2m_s-1)!!)^2}\left(a_{-1}(s) - \frac{P'(s)m_s(m_s+1)}{2P(s)}\right)\la^{m_s} + \dots \  &m_s&\in\Z\label{eq:det_asympt_even},\\
    \end{align}
where the dots denote the terms of degree in $\lambda$ less than $m_s$. 

From that it follows that $m_s\in \Z$ and that \eqref{eq:residue_formula} holds.

To show the claim, we compute the highest term of the polynomial $\Delta(s,\la)$ by 
subtracting a multiple of the second column from the first one to make  $(1,1)$ entry zero and then expanding the determinant along the first row. 

For example, in the first step, we obtain    
\begin{equation}\label{eq:monodromy_free_det}
\Delta(s,\la)=
-2m_s\begin{vmatrix}
a_0(s,\la)-\frac{a_{-1}(s)}{2m_s}a_{-1}(s) & 2\cdot(2m_s-1) & 0 & \cdots & 0 \\
a_1(s,\la)-\frac{a_{-1}(s)}{2m_s}a_{0}(s,\la)  & a_{-1}(s) & 3\cdot(2m_s-2) & \cdots & 0 \\
 \\
\vdots & \ddots & \ddots & \ddots & \ddots & 0 \\
\vdots & \ddots & \ddots & a_0(s,\la) & a_{-1}(s) & 2m_s\cdot1 \\
a_{2m_s-1}(s,\la)-\frac{a_{-1}(s)}{2m_s}a_{2m_s-2}(s,\la)& a_{2m_s-3}(s,\la) & \cdots & a_{1}(s,\la) & a_{0}(s,\la) & a_{-1}(s)
\end{vmatrix}.
\end{equation}

Denote the polynomial in $\lambda$ at the position $(1,1)$ after $l$ such steps by $p_{l}(s,\la)$. Then 
$$p_0(s,\la)=a_{-1}(s), \qquad p_{2m_s} =(-1)^{2m_s}((2m_s)!)^{-2}\Delta(s,\la).$$ 
We have a recurrence relation
\begin{equation}\label{eq:no_monodromy_det_rec}
p_{l+1}(s,\la) =a_{l}(s,\la) - \mathop{\sum}_{j=0}^{l}\frac{a_{-1+j}(s,\la)p_{l-j}(s,\la)}{(l-j+1)(2m_s-l+j)}.
\end{equation}
Let 
$$p_l(s,\la) = c_l \la^{d_l} + \dots,$$ 
where the dots denote the terms of degree in $\lambda$ less than $d_l$.

Recall that $a_{j}(s,\la)$ are linear in $\la$ and $a_{-1}(s,\la)$ does not depend on $\la$.

By induction on $l$, we obtain 
$$
d_{2l} = l,\; d_{2l+1} = l+1,\; \forall l>0.
$$
Then, taking the highest terms in the recurrence \eqref{eq:no_monodromy_det_rec}, for all $l>0$, and keeping in mind \eqref{eq:ajs}, we obtain a recursion
\begin{align}
c_{2l+1} &= - \frac{P(s)c_{2l-1}}{2l(2m_s-2l+1)},\\
c_{2l+2} &= -\left(\frac{a_{-1}(s)c_{2l+1}}{(2l+2)(2m_s-2l-1)} +\frac{P(s)c_{2l}}{(2l+1)(2m_s-2l)} + \frac{P'(s)c_{2l-1}}{2l(2m_s-2l+1)}\right).
\end{align}

That immediately implies the formula \eqref{eq:det_asympt_odd}.

For $c_{2l+2}$, we compute

\begin{multline}\label{eq:even_c}
    c_{2l+2} = \frac{(-1)^{l+1}P(s)^{l+1}(m_s-l-1)!}{(2l+1)!!2^{l+1}}\\\left( \mathop{\sum}\limits_{j=-1}^{l}\frac{(2(l-j)-1)!!}{(2m_s-1)!!(m_s-l+j)!}\left(\frac{(2(m_s-l+j)-1)!!a_{-1}(s)}{(l-j)!} - \frac{2P'(s)}{P(s)}\frac{(2(m_s-l+j)+1)!!}{(l-j-1)!}\right)\right),
\end{multline}
where by definition we set $(-1)!!=1, \frac{1}{(-1)!}=0$.

For $l = m_s-1$ the equation \eqref{eq:even_c} becomes
\begin{multline}
    c_{2m_s} = \frac{(-1)^{m_s}P(s)^{m_s}}{2^{m_s}((2m_s-1)!!)^2}\mathop{\sum}\limits_{j=0}^{m_s}\frac{(2(m_s-j)-1)!!}{j!}\left(\frac{a_{-1}(s)(2j-1)!!}{(m_s-j)!} - \frac{2P'(s)(2j+1)!!}{P(s)(m_s-j-1)!}\right) = \\ 
    = \frac{P(s)^{m_s}}{((2m_s-1)!!)^2}\mathop{\sum}\limits_{j=0}^{m_s}\left(a_{-1}(s)\binom{-\frac{1}{2}}{m_s-j}\binom{-\frac{1}{2}}{j} + \frac{4P'(s)}{P(s)}(m_s-j)(j+1)\binom{-\frac{1}{2}}{m_s-j}\binom{-\frac{1}{2}}{j+1}\right).
\end{multline}

For a series $f(x) =\displaystyle\sum_{j=0}^{\infty}f_jx^j$ in $x$, denote $\{x^j\}[f(x)] := f_j$. Note that 
$$
\mathop{\sum}\limits_{j=0}^{m_s}\binom{-\frac{1}{2}}{m_s-j}\binom{-\frac{1}{2}}{j} = \{x^{m_s}\}\Big[\frac{1}{\sqrt{1+x}}\cdot \frac{1}{\sqrt{1+x}}\Big] =\{x^{m_s}\}\Big[\frac{1}{1+x}\Big] = (-1)^{m_s}.
$$
\begin{multline}
\mathop{\sum}\limits_{j=0}^{m_s}(m_s-j)(j+1)\binom{-\frac{1}{2}}{m_s-j}\binom{-\frac{1}{2}}{j+1} = \{x^{m_s-1}\}\Big[\left(\frac{1}{\sqrt{1+x}}\right)^\prime\cdot \left(\frac{1}{\sqrt{1+x}}\right)^\prime\Big] =\\=\{x^{m_s-1}\}\Big[\frac{1}{4(1+x)^3}\Big] = \frac{(-1)^{m_s+1}m_s(m_s+1)}{8}.
\end{multline}

This proves \eqref{eq:det_asympt_even}.
\end{proof}

We remark that if $m_s = 1$ and the operator $L$ has no monodromy at $x=s$, then the operator $L-\la P$ has no monodromy around $x=s$ for any $\la$ if and only if \eqref{eq:residue_formula} holds.

When $L$ is $\la$-monodromy free, one can give the following heuristic argument for the formula for the residue \eqref{eq:residue_formula}. It is easy to see the formula for a simple pole, namely, for $m_s = 1$.  All other cases can be obtained by a confluence of a triangular number (which is prescribed by Proposition \ref{prop:mf_det}) of such simple poles, see Corollary \ref{cor:mf_limit}. Note, however, that our proof of Theorem \ref{thm:DO_to_populations} and Corollary \ref{cor:mf_limit} uses Proposition \ref{prop:mf_residue}.

The next standard lemma describes functions in the kernel of a monodromy free Schr\"odinger operator as a (square root of) rational function.

\begin{lemma}\label{lemma:rat_ker}
    Let $L=\partial^2-U$ be a Fuchsian operator 
    such that the monodromy around any pole of $U$ is $\pm\mathbb{1}$. 
    For $s\in \C$, let $m_s\in\frac12\Zp$ be the modified exponent of $L$ at $x=s$. 
    Then any $\psi\in \Ker(L)$ has the form
    $$
    \psi(x) = c\prod_{s\in \C}(x-s)^{\mu_s},
    $$
    where $\mu_s \in \{-m_s, m_s{+}1\} \subset \frac{1}{2}\Z$, $\mu_s$ is non-zero only for a finite number of points and $c\in \C$. 

    For any $s\in \C$ there exists a unique up to multiplication by a non-zero complex number solution $\psi$ with $\mu_s=m_s+1$. In particular, for a generic function in $\Ker(L)$, we have $\mu_{s} = -m_s$ for any $s$ such that $m_s \neq 0$.
\end{lemma}\qed

\subsection{Darboux transformations}

Let $\psi \in  \Ker(L)\backslash\{0\}$. There is a factorization
$$
L = \frac{1}{P}\Big(\pd + \ln'(\psi)\Big)\Big(\pd - \ln'(\psi)\Big)=  \frac{1}{\sqrt{P}}\Big(\pd + \ln'(\sqrt{P}\psi)\Big)\frac{1}{\sqrt{P}}\Big(\pd - \ln'(\psi)\Big).
$$

\begin{defi}

    The operator
    \begin{equation}\label{eq:def_Darboux}
        L^{\psi} := \frac{1}{\sqrt{P}}\Big(\pd - \ln'\big(\psi\big)\Big)\frac{1}{\sqrt{P}}\Big(\pd + \ln'\big(\sqrt{P}\psi\big)\Big)= \frac{1}{P}\Big(\pd - \ln'\big(\sqrt{P}\psi\big)\Big)\Big(\pd + \ln'\big(\sqrt{P}\psi\big)\Big),
    \end{equation}
    is called Darboux transformation of $L$ with respect to $\psi$. We denote the potential of $L^{\psi}$ by $U^{\psi}$.
\end{defi}

Denote $$D_{\psi} = \frac{1}{\sqrt{P}}\Big(\pd - \ln'(\psi)\Big).$$
A few simple properties of Darboux transformations are given in the lemma.

\begin{lemma}\label{lemma:Darboux_prop}
    We have
    \begin{enumerate}
        \item $L^{\psi}D_{\psi} = D_{\psi}L$, \label{lemma:Darboux_prop1}
        \item If $\phi \in \Ker(L-\la)$, then $D_{\psi}(\phi) \in \Ker(L^{\psi}-\la)$.\label{lemma:Darboux_prop2} 
        \item For any $\la\neq 0$, the map $D_{\psi}:\ \Ker(L-\la)\to\Ker(L^{\psi}-\la)$ is an isomorphism.\label{lemma:Darboux_prop3}\qed
    \end{enumerate}
\end{lemma}

    Note that in general  the potential $U^{\psi}$  of $L^{\psi} = P^{-1}(\pd^2 - U^{\psi})$is not necessarily rational.

\begin{lemma}\label{lemma:Darboux_MF}
    If $L$ is $\la$-monodromy free Fuchsian Schr\"odinger operator and $\psi\in \Ker(L)\backslash \{0\}$, then $L^{\psi}$ is $\la$-monodromy free Fuchsian Schr\"odinger operator.
\end{lemma}
\begin{proof}
    By Lemma \ref{lemma:rat_ker} and \eqref{eq:def_Darboux}, the potential $U^{\psi}$ of $L^{\psi} = \frac{1}{P}(\pd^2 - U^{\psi})$ is rational and Fuchsian.

    For any $\la\neq 0$ the operator $\partial^2-U^{\psi}-\lambda P$ has monodromy $\pm \mathbb{1}$  by  Lemma \ref{lemma:Darboux_prop}, Part \ref{lemma:Darboux_prop3}. 

    For $\la = 0$ the statement follows by continuity. Namely, the condition on $L^{\psi}$ to be $\la$-monodromy free is a set of algebraic conditions $\Delta(s,\la)=0$ on the coefficients of expansion of $U^{\psi}-\la P$, see  Proposition \ref{prop:mf_det}. Given $s\in\C$, the determinant $\Delta(s,\la)$ vanishes for any $\la\neq 0$, therefore it vanishes for $\la = 0$.
\end{proof}

\begin{remark}
    There exist monodromy free Fuchsian Schr\"odinger operators such that their Darboux transformations are not monodromy free.
\end{remark}
\begin{example}
    Consider an operator
\begin{equation}
    L = \pd^2 -\frac{2}{x^2(x-1)^2},
\end{equation}
We have $P(x) = 1$ and $\Ker(L) = \langle x-2+\frac{1}{x} , (1-x)-2+\frac{1}{1-x}\rangle$. In particular, $L$ is monodromy free. 

Let $\psi = x-2+\frac{1}{x}$. Then
$$
L^{\psi} = \pd^2 -\frac{2(x+2)}{x(x-1)^2},\qquad \Ker(L^{\psi}) = \Big\langle\frac{x}{2 (x-1)^2}, \ \frac{x^4-6 x^3+18 x^2-3}{6 (x-1)^2}-\frac{2 x \log (x)}{(x-1)^2}\Big\rangle. $$ 
The operator $L^{\psi}$ is Fuchsian but not monodromy free. In particular, $L$ is not $\lambda$-monodromy free.

Moreover, the Darboux transformation $(L^{\psi})^{\td \psi}$ of $L^{\psi}$, where $\td\psi = \frac{x^4-6 x^3+18 x^2-3}{6 (x-1)^2}-\frac{2 x \log (x)}{(x-1)^2}$, is a Schr\"odinger operator with a non-rational potential and it is not Fuchsian.
\end{example}

\subsection{The exponents of the Darboux transformed operators}

The next lemma computes the dependence of the exponents of the Darboux transformed operator on the choice of an element in $\Ker(L)$.

Let 
\begin{equation}\label{eq:roots_P}
P(x) = \prod_{j=1}^{n}(x-z_j)^{k_j}, 
\end{equation}
where $k_j\in \Z_{>0}$ for all $j$.

\begin{lemma}\label{lemma:DT_general} Let $L$ be a $\lambda$-monodromy free operator.
    Let $\psi = \prod_{s\in \C}(x-s)^{\mu_s} \in \Ker(L)$.  Then $L^{\psi} = P^{-1}(\pd^2 - U^{\psi})$, where
    $$
    U^{\psi} = \sum_{j=1}^n\left(\frac{\tilde{m}_{z_j}(\tilde{m}_{z_j}+1)}{(x-z_j)^2} + \frac{\al_j}{x-z_j}\right)+ \sum_{s\in \C: P(s)\neq 0} \frac{\tilde{m}_s(\tilde{m}_s + 1)}{(x-s)}\left(\frac{1}{x-s} +\frac{P'(s)}{2P(s)}\right),
    $$
    and
    \begin{equation}\label{eq:exponents_DT}
    \tilde{m}_{z_j} = \max\left(\frac{k_j}{2}+\mu_{z_j}, \ -1-\frac{k_j}{2}-\mu_{z_j}\right), \qquad 
    \tilde{m}_s = \begin{cases}
        m_s-1, &\mu_s = -m_s,\\
        m_s+1, & \mu_s=m_s+1.
    \end{cases}
    \end{equation}
    
    Moreover, the residues $\al_j$ satisfy
    $$
    \sum_{j=1}^n \al_j + \sum_{s\in \C: P(s)\neq 0} \frac{\tilde{m}_s(\tilde{m}_s + 1)P'(s)}{2P(s)} = 0.
    $$
\end{lemma}
\begin{proof}
    The leading terms of the expansion of $U^{\psi}$ at $s$ are obtained by a direct computation.  

    By Lemma \ref{lemma:Darboux_MF} the operator $L^{\psi}$ is $\la$-monodromy free. Therefore the residues of $U^{\psi}$ at points $s$ such that $P(s)\neq 0$ are computed by Proposition \ref{prop:mf_residue}. 
    
    The property on residues $\al_j$ holds since $\underset{x=\infty}{\mathrm{Res}}(U^{\psi})=0$.
\end{proof}

\begin{lemma}\label{lemma:exp_obst}
    Let $L$ be a $\la$-monodromy free operator. Then for any zero $z_j$ of $P$ the corresponding modified exponent $m_{z_j}$ does not belong to $-\frac{1}{2}+ (\frac{k}{2}+ 1)\Z_{>0}$. 
\end{lemma}
\begin{proof}
     Take a generic element $\psi\in \Ker(L)$ and apply the Darboux transformation. Then if $m_{z_j} > \frac{k_j+1}{2}$ by formula \eqref{eq:exponents_DT}  we have $\tilde{m}_{z_j} = m_{z_j}-\frac{k}{2}-1$. In this case we can reduce $\tilde{m}_{z_j}$. If $m_{z_j} = \frac{k_j+1}{2}$ after Darboux transformation we obtain $\tilde{m}_{z_j} = -\frac{1}{2}$. This contradicts part 1 of Proposition \ref{prop:mf_det}. Alternatively, if $m_{z_j} = \frac{k_j+1}{2}$, we have
    $$
    \Delta(z_j, \la) = (-1)^{k_j +1}((k_j +1)!)^2\la + c\not\equiv 0,
    $$
    (cf. part 2 of Proposition \ref{prop:mf_det}). This contradicts the operator being $\la$-monodromy free. 
\end{proof}

In the following two lemmas we discuss the exponents of a $\la$-monodromy free operator after a sequence of generic Darboux transformations. 

\begin{lemma}\label{lemma:ms01}
    Let $L$ be a $\la$-monodromy free operator. Then by a finite sequence of Darboux transformations $L$ can be brought to an operator such that for all $s\in \C  $ with $P(s)\neq 0$ the modified exponents $m_s\in\{0,1\}$.
\end{lemma}
\begin{proof}
We use induction on $\displaystyle\max_{s\in\C:P(s)\neq 0}(m_s)$. Take a generic $\psi(x) \in \Ker(L)$, and apply the Darboux transformation. Then by formula \eqref{eq:exponents_DT}, we reduce all non-zero $m_s$ for $s$ such that $P(s)\neq 0$ by one, and create new poles $\tilde s$ such that $P(\td s)\neq 0$ and $m_{\td s}=1$.
\end{proof}

\begin{lemma}\label{lemma:m0}
    Let $L$ be a $\la$-monodromy free operator such that for any $s\in \C:P(s)\neq 0$, $m_s \in \{0,1\}$. Then by a finite sequence of Darboux transformations $L$ can be brought to an operator with $m_s\in\{0,1\}$ for all $s\in \C$ such that $P(s)\neq 0$ and $m_{z_j}\in \{0,\frac{1}{2},1,\dots, \frac{k_j}{2}\}$, $j = 1,\dots, n$.
\end{lemma}
\begin{proof}
    Given an operator, let $N=\displaystyle\max\left(\max_{j=1,\dots,n}\left(m_{z_j}-\frac{k_j}{2}\right), 0\right)$. We claim that if $N>0$ we can reduce it by applying a Darboux transformation. Indeed, just as in the proof of Lemma \ref{lemma:ms01} take a generic element $\psi\in \Ker(L)$ and apply the Darboux transformation. Then by formula \eqref{eq:exponents_DT} we have $\tilde{m}_s \in \{0,1\}$ for all $s$ such that $P(s)\neq 0$. For any $j$ such that $m_{z_j} \leq \frac{k_j}{2}$ we have $\tilde{m}_{z_j} = \frac{k_j}{2}-m_{z_j} \in \{0,\frac{1}{2},\dots, \frac{k_j}{2}\}$. For any $j$ such that $m_{z_j} > \frac{k_j}{2}$ we have $m_{z_j} > \frac{k_j+1}{2}$ by Lemma \ref{lemma:exp_obst}. In this case we have $\tilde{m}_{z_j} = m_{z_j} - 1 -\frac{k_j}{2} < m_{z_j}$, therefore the number $N$ for the Darboux transformed operator is smaller.
\end{proof}

In the case $P(x) = 1$ an analogue of Lemma \ref{lemma:ms01} is proven and used in  \cite{DG86}.

\section{Populations and  $\la$-monodromy free operators}\label{sec:Pop_and_DO}
In this section we show that any $\la$-monodromy free non-monic Schr\"odinger operator can be obtained from a point in a population of critical points of  an $\sLh$-master function. We apply this relation to classify $\la$-monodromy free operators for the case $P(x) = x^k$ for $k\in \Zp$.

\subsection{Reproductions and Darboux transformations}
Let $\bs T=(T_0,T_1)$ be a pair of polynomials and $P=T_0T_1$.

Let $\bs y$ be a pair of polynomials. Introduce Fuchsian Schr\"odinger operators
\begin{align}\label{eq:LfromYT}
L_j(\bs{y}, \bs{T}) &= P^{-1}(\pd+\ln'(\phi_j))(\pd - \ln'(\phi_j)),
\intertext{where}
\phi_0 &= \frac{y_0}{\sqrt{T_0}y_1}, \qquad \phi_1 = \frac{y_1}{\sqrt{T_1}y_0},
\end{align}
cf. \cite{GLVW22}, \cite{MMR25}.

Let $\bs y$ be in a population of critical points corresponding to $\bs T$, see Section \ref{subsec:populations}.
For $j=0,1$, we describe the change in operators $L_j(\bs{y},\bs{T})$ with respect to the reproduction (see Definition \ref{defi:reproduction}) in terms of Darboux transformations.

Note that for $j=0,1$,
\begin{equation}\label{eq:inv_do}
    {L}_j(\bs{y}^{[j]},\bs{T}) = {L}_j(\bs{y}, \bs{T}).
\end{equation}

If $\bs y=(y_0,y_1)$ is in a population of critical points corresponding to $\bs T=(T_0,T_1)$, then  $(y_1,y_0)$ is in a population of critical points corresponding to $(T_1,T_0)$. Now we have the following important lemma.

\begin{lemma}\label{lemma:Darboux_switch} We have
    \begin{align}
     L_{1}(\bs{y}, \bs{T})^{\phi_1} & =L_0(\bs{y},\bs{T})= L_{1}((y_1,y_0), (T_1,T_0)),\\
     \qquad L_0(\bs{y},\bs{T})^{\phi_0} & =L_{1}(\bs{y}, \bs{T})=L_{0}((y_1,y_0), (T_1,T_0)).
    \end{align}
    
\end{lemma}
\begin{proof}
    The lemma is proved by a direct computation.
\end{proof}

Note that given $\bs{y}, \bs{T}$, the set of all $\bs{y}^{[j]}$ (together with $\bs{y}$) is in a bijective correspondence with the projectivization of $\Ker(L_j(\bs{y},\bs{T}))$.
\begin{cor}\label{cor:DT_pop}
    A Darboux transformation of an operator $L_{j}(\bs{y},\bs{T})$ corresponding to a point $\bs{y}$ in a population is an operator $L_{j}((y_1,y_0),(T_1,T_0))$ also corresponding  to a point $(y_1,y_0)$ in a (possibly different) population. \qed
\end{cor}

\begin{cor}\label{cor:Darboux_rep} We have
    \begin{align}
        L_1(\bs{y}^{[0]},\bs{T}) = L_0(\bs{y},\bs{T})^{\phi_0^{[0]}} = \Big(L_1(\bs{y},\bs{T})^{\phi_1}\Big)^{\phi_0^{[0]}}, \qquad \phi_0^{[0]} = \frac{y_0^{[0]}}{\sqrt{T_0}y_1},\\
        L_0(\bs{y}^{[1]},\bs{T}) = L_1(\bs{y},\bs{T})^{\phi_1^{[1]}} = \Big(L_0(\bs{y},\bs{T})^{\phi_0}\Big)^{\phi_1^{[1]}},\qquad \phi_1^{[1]} = \frac{y_1^{[1]}}{\sqrt{T_1}y_0}.
    \end{align}\qed
\end{cor}

\begin{prop}\label{prop:la_mf_fromcp}\cite{GLVW22} Let $\bs y$ be in a population of critical points corresponding to $\bs T$. Then
    the operators $L_j(\bs{y}, \bs{T})$ are $\la$-monodromy free. 
\end{prop}
\begin{proof}
    For the case when $\bs{y}$ represents a critical point in the population, the proof is given in \cite{GLVW22}. For an arbitrary pair $\bs y$ in the population the proposition follows by continuity as the pairs corresponding to critical points are dense in the population.
\end{proof}

We note that by Proposition \ref{prop:la_mf_fromcp} for each $P$ we immediately obtain a large set of $\la$-monodromy free operators. Indeed, factorize $P=T_0T_1$. 
Then there exists a population of critical points generated by $(1,1)$ with respect to $(T_0,T_1)$. The points of such population are given by explicit determinants, see \eqref{eq:theta_Wr}. 
Then all corresponding differential operators are $\la$-monodromy free.

We remark that due to Proposition \ref{prop:population_from_sf}, the differential operators corresponding to super-fertile pairs of polynomials are the same as the operators corresponding to pairs in the populations of critical points.
\begin{cor}
    For any super-fertile tuple $\bs{y}^{\circ}$ corresponding to $\bs{T}^{\circ} = (T_0^{\ci}, T_1^{\ci})$ there exists a pair $\bs{y}$ in a population of critical points corresponding to a pair $\bs{T}$ such that $L_1(\bs{y}^{\ci},\bs{T}^{\ci}) = L_1(\bs{y},\bs{T})$. In particular the operator $L_1(\bs{y}^{\ci},\bs{T}^{\ci})$ is $\la$-monodromy free.
\end{cor}
\begin{proof}
    By Proposition \ref{prop:population_from_sf} the operator has the form $L_{1}(\bs{y}, \bs{T})$, where $\bs{y}$ is a point in a population corresponding to $\bs{T}$. Moreover, $P^\ci=P = T_0T_1$. Then the statement follows from Proposition \ref{prop:la_mf_fromcp}.
\end{proof}

\subsection{From $\la$-monodromy free operators to populations}\label{subsec:from_lmf_to_pop}

In this section we show the  converse to Proposition \ref{prop:la_mf_fromcp}. Namely, we show that any $\la$-monodromy free operator has the form $L_1(\bs{y},\bs{T})$, where a pair $\bs{y}$ is in a population of critical points corresponding to a pair $\bs T$ such that $T_0 T_1 = P$. We start from the case of the operator corresponding to a critical point. 

Recall notation \eqref{eq:roots_P}.
\begin{lemma}\label{lemma:op_to_BA}
    Let $L$ be a $\la$-monodromy free operator. Assume that the modified exponents satisfy $m_s \in  \{0,1\}$ for all $s\in \C$ such that $P(s)\neq 0$ and $m_{z_j} \in \{0,\frac{1}{2},\dots, \frac{k_j}{2}\},\; j = 1,\dots,n$. Then there exists a critical point $\bs{y}$ of $\sLh$-master function corresponding to $T_0(x) = \prod_{j=1}^n (x-z_j)^{k_j-2m_{z_j}},\; T_1(x) = \prod_{j=1}^{n}(x-z_j)^{2m_{z_j}}$,
    such that 
    $
    L = L_1(\bs{y},\bs{T}),
    $
    as in \eqref{eq:LfromYT}.
\end{lemma}
\begin{proof}
     Let $s_1,\dots,s_{d_0}$ be the  poles of $L$ that are not zeroes of $P$. Since the operator $L$ is $\la$-monodromy free, it has the form $L = P^{-1}(\pd^2 - U)$ where 
     $$
     U = \sum_{j=1}^n \left(\frac{m_{z_j}(m_{z_j}+1)}{(x-z_j)^2}+ \frac{\al_j}{x-z_j}\right) + \sum_{i=1}^{d_0}\frac{2}{x-s_i}\left(\frac{1}{x-s_i} + \sum_{j=1}^{n}\frac{k_j}{2(s_i-z_j)}\right),
     $$
    where $\al_j\in\C$, $j=1,\dots,n$, see Proposition \ref{prop:mf_residue}.

     Set $\displaystyle y_0=\prod_{i=1}^{d_0}(x-s_i)$. Then by Lemma $\ref{lemma:rat_ker}$ we obtain $\Ker(L) = \Ker(PL)= \left<\frac{y_1}{\sqrt{T_1}y_0},\frac{\tilde{y}_1}{\sqrt{T_1}y_0}\right>$, where $y_1,\tilde{y}_1$ are polynomials without common zeros and without zeroes in $\{z_1,\dots, z_n,s_1,\dots, s_{d_0}\}$. We claim that the pair $\bs{y} = (y_0,y_1)$ is a critical point corresponding to $\bs{T} = (T_0,T_1)$. 
     
    Since $L$ has no term of the form $q(x)\pd$ we have
    \begin{equation}\label{eq:wr11tmp}
        \Wr(y_1,\tilde{y}_1) = T_1y_0^2.
    \end{equation}

    Conjugation of $PL$ by $\sqrt{T_1}y_0$ is an operator with kernel $\langle y_1, \tilde{y}_1\rangle$ given by 
    \begin{multline}
        \pd^2 - 2\left(\mathop{\sum}\limits_{i=1}^{d_0}\frac{1}{x-s_i}+ \sum_{j=1}^{n}\frac{m_{z_j}}{x-z_j}\right)\pd+\sum_{1\leq i<i^{\prime}\leq d_0}\frac{2}{(x-s_i)(x-s_{i^{\prime}})} {+} \sum_{1\leq j<j^{\prime}\leq n}\frac{2m_{z_j}m_{z_{j^{\prime}}}}{(x-z_j)(x-z_{j{\prime}})} {+}\\{+} 2
        \left(\sum_{j=1}^{n}\frac{m_{z_j}}{x-z_j}\right)\left(\mathop\sum\limits_{i=1}^{d_0}\frac{1}{x-s_i}\right) - \sum_{i=1}^{d_{0}}\left(\frac{1}{x-s_i}\sum_{j=1}^{n}\frac{k_j}{s_i-z_j}\right) -\sum_{j=1}^{n}\frac{\al_{j}}{x-z_j}.
    \end{multline}

    Apply this operator to $y_1$ and take the residue at $x= s_i$. Dividing the result by $y_1(s_i)$, we obtain
    \begin{equation}\label{eq:BA2tmp}
    \sum_{i^{\prime}:i^{\prime}\neq i}\frac{2}{s_i-s_{i^{\prime}}}-\sum_{l=1}^{d_1}\frac{2}{s_i-t_l} - \sum_{j=1}^n\frac{k_j-2m_{z_j}}{s_i-z_j} = 0,\;\; i=1,\dots,n,
    \end{equation}
    where $t_1,\dots, t_{d_1}$ are zeroes of $y_1(x)$.
    
    The system \eqref{eq:wr11tmp}, \eqref{eq:BA2tmp} is equivalent to the system of equations on a critical point, cf. Proposition \ref{prop:cp_wr}. 
\end{proof}

For given $\bs{y}$ in a population $\mathcal{P}$ let us denote by $[\bs{y}^{[i]}]$ the closure of the set of all possible pairs obtained by the reproduction of $\bs{y}$ in the $i$-th direction. The variety $[\bs{y}^{[i]}]$ is isomorphic to $\mathbb{CP}^1$.

\begin{theorem}\label{thm:DO_to_populations}
    Let $L$ be a $\la$-monodromy free operator. Then there exists a unique line $[\bs{y}^{[1]}]$ in a unique population corresponding to a pair of polynomials $\bs T$ such that $L = L_{1}(\bs{y},\bs{T})$, and $P = T_0T_1$.
\end{theorem}
\begin{proof}
    Lemma \ref{lemma:ms01} and Lemma \ref{lemma:m0} show that $L$ can be obtained by a finite sequence of Darboux transformations from an operator satisfying the conditions of Lemma \ref{lemma:op_to_BA}. Therefore $L$ is obtained by a finite sequence of Darboux transformations from an operator of the form $L_1(\bs{y},\bs{T})$, where $\bs{y}$ is a critical point. Then the existence part follows from Corollary \ref{cor:DT_pop}. The existence of a one-dimensional family of pairs of polynomials follows from \eqref{eq:inv_do}.

    It remains to show that such a family is unique. Assume that there are some other pairs $\bar{\bs{y}}, \bar{\bs{T}}$ such that $L = L_1(\bar{\bs{y}}, \bar{\bs{T}})$. Then by adjusting the representatives $\bs{y}$ and $\bar{\bs{y}}$ on the corresponding lines we obtain 
    $$\frac{\bar{y}_1}{\sqrt{\bar{T}_1}\bar{y}_0} = \frac{y_1}{\sqrt{T_1}y_0}.$$
    
    There exist polynomials $f_0,f_1,\bar{f}_0,\bar{f}_1$ such that the pairs $f_0,\bar{f}_0$ and $f_1,\bar{f}_1$ are coprime and such that $f_iy_i = \bar{f_i}\bar{y}_i$, $i = 0,1$. This implies 
    \begin{equation}\label{eq:op_to_populations_pf1}
    \frac{T_1f_1^2}{f_0^2} = \frac{\bar{T}_1\bar{f}_1^2}{\bar{f}_0^2},\;\; \frac{T_0f_0^2}{f_1^2} = \frac{\bar{T}_0\bar{f}_0^2}{\bar{f}_1^2}.
    \end{equation}
    Since the denominators of both sides of equations \eqref{eq:op_to_populations_pf1} are coprime both sides of both equations are polynomials. By Lemma \ref{lemma:sf_from_population} we obtain a superfertile pair $\bs{y}^{\ci}=(f_0y_0,f_1y_1) = (\td f_0\td y_0,\td f_1\td y_1)$ corresponding to the pair $\bs{T}^{\ci} = \left(\frac{f_0^2T_0}{f_1^2},\frac{f_1^2T_1}{f_0^2}\right) = \left(\frac{\bar{f}_0^2\bar{T}_0}{\bar{f}_1^2},\frac{\bar{f}_1^2\bar{T}_1}{\bar{f}_0^2}\right)$. But then by Proposition \ref{prop:population_from_sf}, we have $\bar{\bs{T}}=\bs{T}$ and $\bar{\bs{y}}=\bs{y}$.
\end{proof}

\begin{cor}\label{cor:mf_limit}
    Any $\la$-monodromy free operator can be obtained as a limit of $\la$-monodromy free operators with modified exponents $m_{z_j} \leq k_j$, 
    $j = 1,\dots, n$ and  $m_s \in \{0,1\}$ for all other $s\in \C$.
\end{cor}
\begin{proof}
    Follows from Theorem \ref{thm:DO_to_populations}, since in any population the pairs of polynomials representing critical points are dense. 
\end{proof}

\subsection{The case $P(x) = x^k$.  }\label{subsec:class_one_point}

In this section we fix $$P(x) = x^{k}, \qquad k\in \Z_{\geq 0},$$
and apply the results of previous sections to give a full classification of $\la$-monodromy free operators.

\begin{theorem}\label{thm:MV_crit}\cite{MV14}
    Let $T_0(x) = x^{2m},\; T_1(x)=x^{2l}$, where $2m,2l\in \Zp$.
    
    Then there exists a unique population of critical points corresponding to $\bs{T}$. This population is the population of trivial critical point $(1,1)$. 
 \qed
\end{theorem}

In fact, the $\sLh$-population is unique if $T_0=(x-a)^{2m}$ and $T_1=(x-a)^{2l}$, where $a\in\C$, and in all other cases we expect to have infinitely many populations.

\begin{theorem}\label{thm:op_clas}
    Let $L$ be a $\la$-monodromy free operator and let $m_0$ be the modified exponent at $x=0$. Let $m \in \frac{1}{2}\Zp,\; m \leq \frac{k}{4}$,  be the unique number such that either $m_0 - m \in \left(\frac{k}{2}+1\right)\Zp$ or $m_0 - (\frac{k}{2}-m) \in \left(\frac{k}{2}+1\right)\Zp$. Then $L$ is obtained  from the operator $x^{-k}\left(\pd^2 - \frac{m(m+1)}{x^2}\right)$  by a finite sequence of Darboux transformations.
\end{theorem}
\begin{proof}
    The number $m$ exists by Lemma \ref{lemma:exp_obst}. Clearly it is unique.

    By Theorem \ref{thm:DO_to_populations} there exists a point $\bs{y}$ such that $L = L_1(\bs{y}, \bs{T})$. Here $\bs{y}$ is in a population corresponding to $\bs{T} = (T_0,T_1)$ such that $T_0T_1 = P = x^k$. Therefore, we have $T_0(x) = x^{2m^{\prime}},\; T_1(x) = x^{2l^{\prime}}$, where $2m^{\prime}, 2l^{\prime}\in \Zp$ and $2m^{\prime}+2l^{\prime} =k$. By Theorem \ref{thm:MV_crit} $\bs{y}$ is in the population of the trivial critical point  $(1,1)$. Therefore by Corollary \ref{cor:Darboux_rep} the operator $L = L_1(\bs{y},\bs{T})$ is obtained by a finite number of Darboux transformations from $L_1((1,1),\bs{T}) = \pd^2 - \frac{m^{\prime}(m^{\prime}+1)}{x^2}$.
    
    If $m^{\prime} \leq \frac{k}{4}$ we set $m = m^{\prime}$. Otherwise we have $l^{\prime} = \frac{k}{2}-m^{\prime} \leq \frac{k}{4}$. In this case we apply one more Darboux transformation to $L_{1}((1,1),\bs{T})$ with respect to $\psi = \frac{1}{\sqrt{T_1}} = x^{-l^{\prime}}$.
    The result is the operator $L_0((1,1),\bs{T}) = \pd^2 - \frac{l^{\prime}(l^{\prime}+1)}{x^2}$ (see Lemma \ref{lemma:Darboux_switch}). Then we set $m = l^{\prime}$.
    
\end{proof}

The case $P(x)=1$ of Theorem \ref{thm:op_clas} can be deduced from \cite{DG86}.

\section{Populations of trivial critical point and change of variables in \text{KdV} hierarchy}\label{sec:KdV}

\subsection{Adler--Moser polynomials}
In the case of $\bs{T}=(1,1)$, the unique population of critical points has been extensively studied, see \cite{AM78, dV17}. The corresponding sequence of polynomials $\{\theta_{n}(x)\}$ (cf. Section \ref{subsec:triv_pop}) is known as Adler--Moser polynomials. 

The following celebrated result of Adler and Moser connects Adler--Moser polynomials to tau-functions of the KdV hierarchy.

Denote by $c_{n+1}$ an integration constant appearing from resolving recurrence \eqref{eq:AM_gen_recc} for $\theta_{n}=\theta_{n+1}^{(0)}$ with $T_{n} \equiv 1$ for all $n$. Then the polynomial $\theta_{n}$ depends on $n$ parameters $c_1,\dots,c_n$.

\begin{theorem}\label{thm:AM}\cite{AM78}
There exists a change of variables $c_j = a_j t_j + p_j(t_1,\dots, t_{j-1})$, where $a_j \in \C^\times$ and $p_j$ are polynomials such that $\theta_n(t_1,\dots, t_n)$ is a tau function of the KdV hierarchy.
\end{theorem}

Explicitly Theorem \ref{thm:AM} means the following. Let $L = \pd^2 - u(x,t_1,t_2,t_3,\dots)$. The flows of KdV hierarchy are given by
\begin{equation}\label{eq:Kdv_H}
    \pd_{t_{j}}(L) = [(L^{\frac{2j-1}{2}})_{+}, L].
\end{equation}

Then Theorem \ref{thm:AM} asserts that the Adler Moser polynomials give solutions of \eqref{eq:Kdv_H} of the form 
\begin{equation}\label{eqn:AM_ops}
L_{n}(t,x)= L_{0}((\theta_{n+1},\theta_{n}),(1,1)) = \pd^2 - 2\ln''(\theta_n).
\end{equation}

The earlier result of \cite{AMM77} implies that any solution $u(x,t_1,t_2,t_3,\dots)$ of KdV that stays rational as a function of $x$ for any $t =(t_1,t_2,\dots)$ has the form of the potential of an operator \eqref{eqn:AM_ops} for some $n$. 
 
\subsection{Change of variables in populations and in KdV}
Note that equations \eqref{eq:cp_wr} possess a nice behavior with respect to polynomial changes of variable.
\begin{lemma}
    Let $f\in \C[x]$ be a non-constant polynomial. Let $\mathcal{P}$ be a population of critical points corresponding to $\bs{T} = (T_0,T_1)$. Then the set $\{(y_0( f(x)),y_1(f(x)))\ 
    |\ (y_0,y_1)\in\mathcal{P}\}$ is a population corresponding to the pair of polynomials $(f^{\prime}(x)T_0(f(x)), f^{\prime}(x)T_1(f(x)))$.
\end{lemma}
\begin{proof}
    Observe that 
    $$
    \Wr(y_i(f(x)), \tilde{y}_i(f(x)))  = f'(x)T_{i}(f(x))y_{i+1}(f(x))^2.
    $$
    There exists a generic pair $\bs{y}= (y_0,y_1)$ in $\mathcal{P}$ such that $y_0$ and $y_1$ do not vanish at $f(z_j), f(\bar{z}_j)$ where $z_j$ are zeroes of $T_0T_1$ and $\bar{z}_j$ are zeroes of $f'(x)$. Then the pair $(y_0(f(x)),y_1(f(x)))$ is generic with respect to $(f^{\prime}(x)T_0(f(x)), f^{\prime}(x)T_1(f(x)))$.
\end{proof}

\begin{cor}\label{cor:diag_pop_from_AM}
    For any $T(x)\in \C[x]$ the population of critical points of $(1,1)$ corresponding to $(T(x),T(x))$ is given by $\{\theta_{n}(\int T(x))\}$, where $\theta_{n}(x)$ are Adler--Moser polynomials.\qed
\end{cor}

The change of variables can also be done on the side of KdV hierarchy. This gives a family of potentials solving a non-monic version of equations \eqref{eq:Kdv_H}.

For a non-zero function $g(x)$ in a variable $x$ and a pseudo-differential operator $D$ denote by $\mathrm{Ad}_{g(x)}(D) = g(x)Dg(x)^{-1}$.

Let $L$ be a differential operator and let $f(\xi)$ be any non-constant rational function. Let
\begin{equation}
\tilde{L} =\mathrm{Ad}_{f'(\xi)^{-1/2}}(L|_{x=f(\xi)}) = \frac{1}{f'(\xi)^2}\left(\pd_\xi^2 - (f'(\xi)^2u(f(\xi))-\frac{1}{2}Sf(\xi))\right).
\end{equation}

Here $Sf(\xi) = \frac{f'''(\xi)}{f'(\xi)}-\frac{3 f''(\xi)^2}{2 f'(\xi)^2}$ is Schwarz derivative of $f$.

\begin{lemma}\label{lemma:KdV_var_change}
    Assume that $L$ satisfies the KdV hierarchy \eqref{eq:Kdv_H}. Then $\tilde{L}$ satisfies
    \begin{equation}\label{eq:nmKdV}
    \pd_{t_j}(\tilde{L}) = [(\tilde{L}^{\frac{2j-1}{2}})_{+},\tilde{L}],\qquad j = 1,2,\dots.
    \end{equation}
\end{lemma}
\begin{proof}
    Let $\varphi(\xi) =(f'(\xi))^{-1/2}$. For $j>1$ we have 
    \begin{equation}
    \pd_{t_j}(\tilde{L}) = \mathrm{Ad}_{\varphi(\xi)}((\pd_{t_j}(L))|_{x=f(\xi)}) =  \mathrm{Ad}_{\varphi(\xi)}(([(L^{\frac{2j-1}{2}})_+, L])|_{x=f(\xi)}) = [\mathrm{Ad}_{\varphi(\xi)}\big((L^{\frac{2j-1}{2}})_+|_{x=f(\xi)}\big), \tilde{L}].
    \end{equation}

    Note that  for any pseudo-differential operator $D$ we have $D_{+} |_{x=f(\xi)} = (D|_{x=f(\xi)})_{+}$ since 
    $$\left(\left(c(x)\pd_x^j\right)|_{x=f(\xi)}\right)_{+} = \left(c(f(\xi))\left(\frac{1}{f'(\xi)}\pd_\xi\right)^{j}\right)_{+} = \begin{cases}
        0,\hspace{82pt} j<0,\\
        c(f(\xi))(\frac{1}{f'(\xi)}\pd_\xi)^{j},\;\; j\geq 0
    \end{cases} = \left(c(x)\pd_x^j\right)_{+}\big|_{x=f(\xi)}.$$

    Moreover, for any pseudo-differential operator $D$ in variable $\xi$ and any non-zero rational function $\varphi(\xi)$ we have $\mathrm{Ad}_{\varphi(\xi)}(\xi)D_{+}= (\mathrm{Ad}_{\varphi(\xi)}D)_{+}$. This implies $\mathrm{Ad}_{\varphi(\xi)}((L^{\frac{2j-1}{2}})_+|_{x=f(\xi)}) = (\mathrm{Ad}_{\varphi(\xi)}(L^{\frac{2j-1}{2}}|_{x=f(\xi)}))_+$. 
    
    It remains to note that $L^{\frac{2j-1}{2}}|_{x=f(\xi)} = (L|_{x=f(\xi)})^{\frac{2j-1}{2}}$ and $\mathrm{Ad}_{\varphi(\xi)}((L|_{x=f(\xi)})^{\frac{2j-1}{2}}) = \tilde{L}^{\frac{2j-1}{2}}$.

    For $j = 1$ the statement follows by a direct computation using 
    $$(\tilde{L}^{\frac{1}{2}})_+ = \frac{1}{f'(\xi)}\pd_\xi +  \frac{f''(\xi)}{2f'(\xi)^2},\qquad \pd_{t_1}(\tilde{L}) = -u'(x)|_{x=f(\xi)}.$$
\end{proof}
Note that equation \eqref{eq:nmKdV}  is different from \eqref{eq:Kdv_H}. In \eqref{eq:Kdv_H} the operator $L$ has the form $\partial^2-u(x,t)$ while in \eqref{eq:nmKdV} the operator $\tilde L$ has the form $\frac{1}{P(x)}\pd^2-u(x,t)$.

Let $T(x)$ be an arbitrary polynomial and $f(x)$ be such that $f'(x)=T(x)$. Let $\mathcal{P}(T)$ be the population generated by the trivial critical point $(1,1)$ with respect to $\bs{T} = (T(x),T(x))$.

We apply Lemma \ref{lemma:KdV_var_change} to show that operators $L_0(\bs{y},\bs{T})$ corresponding to the points of the population $\mathcal{P}(T)$ satisfy KdV-type equations.

We fix the KdV times or the integration constants by connecting the points $\bs y\in\mathcal{P}(T)$ to Adler--Moser polynomials $\theta_n$ by 
$\bs y=(\theta_n(f(x)),\theta_{n+1}(f(x)))$, see 
Corollary \ref{cor:diag_pop_from_AM}.

\begin{cor}\label{cor:diag_sol_via_AM}
     Let $\bs{y}\in\mathcal{P}(T)$. Let $\hat{L} = L_0(\bs{y},\bs{T})$. Then for any $j \in \Zp$ we have
    \begin{equation}
    \partial_{t_{j}}(\hat{L}) = \Big[\Big(\hat{L}^{\frac{2j-1}{2}}\Big)_{+}, \hat{L}\Big].
    \end{equation}
\end{cor}
\begin{proof}
    The statement follows by observing that there is $n\in \Zp$ such that

    \begin{multline}
    \hat{L} = \frac{1}{T(x)^2}\left(\pd_x + \ln'\left(\frac{y_1(x)}{\sqrt{T(x)}y_0(x)}\right)\right)\left(\pd_{x} - \ln'\left(\frac{y_1(x)}{\sqrt{T(x)}y_0(x)}\right)\right) = \\
    = \mathrm{Ad}_{\frac{1}{\sqrt{T(x)}}}\left[\frac{1}{T(x)}\left(\pd_{x} + \ln'\left(\frac{y_1(x)}{y_0(x)}\right)\right) \frac{1}{T(x)}\left(\pd_{x} - \ln'\left(\frac{y_1(x)}{y_0(x)}\right)\right)\right] =\\= \mathrm{Ad}_{\frac{1}{\sqrt{f'(x)}}}(L_0((\theta_{n+1},\theta_{n}),(1,1))|_{x\to f(x)}).
    \end{multline}
    Here $f(x)$ is such that $f'(x)=T(x)$. Therefore the corollary follows from Lemma \ref{lemma:KdV_var_change}.
\end{proof}

We do not know if Corollary \ref{cor:diag_sol_via_AM}  describes all rational solutions of \eqref{eq:nmKdV}.

We do not know if there are natural integrable hierarchies corresponding to other populations of critical points.

\medskip

{\bf Acknowledgments.\ }
We thank V. Tarasov for valuable comments and discussions. The first author thanks SISSA and BIMSA for hospitality.

The second author is partially supported by Simons Foundation grant number \#709444.

\end{document}